\documentclass[12pt]{amsart}

\usepackage{amsmath,amssymb,amsfonts,amsxtra,fullpage,xypic,centernot,setspace,lineno,verbatim}

\usepackage{hyperref}

\theoremstyle{plain}
\newtheorem{theorem}{Theorem}[subsection]
\newtheorem{lemma}[theorem]{Lemma}
\newtheorem{proposition}[theorem]{Proposition}
\newtheorem{corollary}[theorem]{Corollary}

\theoremstyle{definition}
\newtheorem{definition}[theorem]{Definition}

\theoremstyle{remark}
\newtheorem{remark}[theorem]{Remark}
\newtheorem{question}{Question}

\numberwithin{equation}{subsection}

\DeclareMathOperator{\Ball}{\operatorname{Ball}}
\DeclareMathOperator{\bbC}{\mathbb{C}}

\DeclareMathOperator{\bbE}{\mathbb{E}}
\DeclareMathOperator{\bbF}{\mathbb{F}}
\DeclareMathOperator{\bbH}{\mathbb{H}}
\DeclareMathOperator{\bbI}{\mathbb{I}}
\DeclareMathOperator{\bbM}{\mathbb{M}}
\DeclareMathOperator{\bbN}{\mathbb{N}}

\DeclareMathOperator{\bbR}{\mathbb{R}}
\DeclareMathOperator{\bbT}{\mathbb{T}}

\DeclareMathOperator{\A}{\mathcal{A}}
\DeclareMathOperator{\Ad}{\operatorname{Ad}}
\DeclareMathOperator{\Aut}{\operatorname{Aut}}
\DeclareMathOperator{\aut}{\operatorname{Aut}}
\DeclareMathOperator{\B}{\mathcal{B}}

\DeclareMathOperator{\COL}{\operatorname{Col}}
\DeclareMathOperator{\D}{\mathcal{D}}

\newcommand{\eps}{\varepsilon}
\DeclareMathOperator{\F}{\mathcal{F}}
\DeclareMathOperator{\fin}{\operatorname{fin}}
\DeclareMathOperator{\Fix}{\operatorname{Fix}}
\renewcommand{\H}{\mathcal{H}}
\DeclareMathOperator{\Homeo}{\operatorname{Homeo}}

\DeclareMathOperator{\id}{\operatorname{id}}

\DeclareMathOperator{\J}{\mathcal{J}}

\DeclareMathOperator{\M}{\mathcal{M}}

\let\slasho=\o
\renewcommand{\o}{\overline}
\renewcommand{\O}{\mathcal{O}}

\DeclareMathOperator{\ROW}{\operatorname{Row}}

\let\slashS=\S
\renewcommand{\S}{\mathcal{S}}

\DeclareMathOperator{\typeI}{\operatorname{I}}
\DeclareMathOperator{\typeII}{\operatorname{II}}

\newcommand{\U}{\mathcal{U}}
\newcommand{\bh}{B(\H)}
\newcommand{\innerprod}[1]{\left\langle #1\right\rangle}
\newcommand{\norm}[1]{\left\|{#1}\right\|}

\newcommand{\cstaralg}{$C^*$-algebra}
\newcommand{\ran}{\operatorname{range}}
\begin{document}

\title{Norming in Discrete Crossed Products}
\author{David R. Pitts} \thanks{This work was supported by grants from the Simons Foundation (DRP \#316952, RRS \#522375).}
\address{University of Nebraska, Lincoln, NE 68588}
\email{dpitts2@unl.edu}
\author{Roger R. Smith}
\address{Texas A\&M University, College Station, TX 77843}
\email{rrsmith@tamu.edu}
\author{Vrej Zarikian}
\address{U. S. Naval Academy, Annapolis, MD 21402}
\email{zarikian@usna.edu}
\subjclass[2020]{Primary: 46L05, 46L40. Secondary: 46L10, 46L55, 46M10}
\keywords{$C^*$-algebras, crossed products, group actions, norming, pseudo-expectations}
\date{\today}
\begin{abstract}
Let $G \curvearrowright \A$ be an action of a discrete group on a unital $C^*$-algebra by $*$-automorphisms. In this note, we give two sufficient dynamical conditions for the $C^*$-inclusion $\A \subseteq \A \rtimes_r G$ to be norming in the sense of Pop, Sinclair, and Smith. As a consequence of our results, when $\A$ is separable or simple, the inclusion $\A \subseteq \A \rtimes_r G$ is norming provided it has a unique pseudo-expectation in the sense of Pitts.
\end{abstract}
\maketitle



\section{Introduction}

Pitts showed that when $\A \subseteq \B$ is a regular MASA inclusion, there exists a unique pseudo-expectation, i.e., a unique UCP (unital completely positive) map $\theta:\B \to I(\A)$ extending the inclusion $\iota:\A \to I(\A)$ of $\A$ into its injective envelope $I(\A)$ \cite[Theorem 3.5]{Pitts2017}. If, in addition, $\theta$ is faithful, meaning that $\theta(x^*x) = 0$ implies $x = 0$, then Pitts \cite[Theorem 8.2]{Pitts2017} proved that the $C^*$-inclusion $\A \subseteq \B$ is norming in the sense of Pop, Sinclair, and Smith \cite[Definition 2.1]{PopSinclairSmith2000}. That is, for all $X \in M_n(\B)$,
\[
    \|X\| = \sup\{\|RXC\|: R \in \Ball(\ROW_n(\A)), ~ C \in \Ball(\COL_n(\A))\}.
\]
Norming is a useful property for $C^*$-inclusions, since it allows one to deduce that certain bounded maps are in fact completely bounded (e.g., \cite[Theorem 2.10]{PopSinclairSmith2000} and \cite[Theorem 1.4]{Pitts2008}). This, in turn, enables one to show that certain isometric isomorphisms of non-self-adjoint operator algebras extend uniquely to $*$-isomorphisms of the generated $C^*$-algebras (e.g., \cite[Corollary 1.5]{Pitts2008}, \cite[Theorem 2.6]{Pitts2008}, \cite[Theorem 3.1.6]{CameronPittsZarikian2013}, and \cite[Theorem 8.4]{Pitts2017}). Pitts and Zarikian identified some other situations when having a faithful unique pseudo-expectation (in the sense of Definition~\ref{uextdefs}\eqref{uextdefs4}) implies norming \cite[Theorem 6.8]{PittsZarikian2015}. On the other hand, they exhibited an example of a non-norming $C^*$-inclusion $\A \subseteq \B$ with a faithful unique pseudo-expectation \cite[Example 6.9]{PittsZarikian2015}. In that example, $\A$ is non-separable. If $\B$ is separable, it remains open whether or not every $C^*$-inclusion $\A \subseteq \B$ with a faithful unique pseudo-expectation is norming \cite[Question 9 in Section 7.1]{PittsZarikian2015}. In this note, we show that if $\A$ is separable and $G \curvearrowright \A$ is an action of a discrete group by $*$-automorphisms, then the $C^*$-inclusion $\A \subseteq \A \rtimes_r G$ is norming provided it has a faithful unique pseudo-expectation. By \cite[Theorem 3.5]{Zarikian2019a} this happens when $G \curvearrowright \A$ satisfies Kishimoto's condition (see Definition~\ref{outercond}\eqref{outercond5}). 

More generally, for arbitrary $\A$ (not necessarily separable), Theorem~\ref{norming_I} and Theorem~\ref{norming_II} below give two different sufficient conditions on the action $G \curvearrowright \A$ that guarantee $\A \subseteq \A \rtimes_r G$ is norming. Interestingly, the proof of Theorem~\ref{norming_I} is $C^*$-algebraic, while the proof of Theorem~\ref{norming_II} is more von Neumann algebraic in character. Moreover, our arguments utilize the connections and subtle differences between outerness properties of the action of a discrete group $G$ on a $C^*$-algebra $\A$, and freeness properties of the related topological action of $G$ on the spectrum $\mathcal{\widehat{A}}$ of $\A$.

\section{Preliminaries}

\subsection{Standing Assumptions and Notation}

All \underline{ambient} $C^*$-algebras will be \underline{unital} unless otherwise specified. Of course, this does not apply to ideals and hereditary subalgebras of $C^*$-algebras. Given a state $\phi$ on the $C^*$-algebra $\A$, the corresponding GNS representation will be denoted $(\pi_\phi,\H_\phi,\xi_\phi)$. The term \textbf{$\mathbf{C^*}$-inclusion} will always mean an inclusion $\A \subseteq \B$ of \underline{unital} $C^*$-algebras with the \underline{same unit}.

A \textbf{discrete $\mathbf{C^*}$-dynamical system} is a triple $(\A,G,\alpha)$, where $\A$ is a (unital) $C^*$-algebra, $G$ is a discrete group, and $\alpha:G \to \Aut(\A)$ is a homomorphism of $G$ into the automorphism group of $\A$. If $(\A,G,\alpha)$ is a discrete $C^*$-dynamical system, then $\A \rtimes_r G$ (resp. $\A \rtimes_f G$) will denote the reduced (resp. full) crossed product of $\A$ by $G$ with respect to $\alpha$.  (We suppress reference to the action $\alpha$ in the notation because we  only consider a single action.)    If, in addition, $\A$ is a von Neumann algebra, then $\A \o{\rtimes} G$ will denote the von Neumann algebra crossed product of $\A$ by $G$ with respect to $\alpha$. Each of these crossed products contains a dense $*$-subalgebra (with respect to the appropriate topology) consisting of finite sums of formal products $ag$, with $a \in \A$ and $g \in G$, subject to the relations
\[
    (a_1g_1)(a_2g_2) = a_1\alpha_{g_1}(a_2)g_1g_2 \text{ and } (ag)^* = \alpha_g^{-1}(a^*)g^{-1}.
\]
There are canonical inclusions $\A \subseteq \A \rtimes_r G$ and (if $\A$ is a von Neumann algebra) $\A \rtimes_r G \subseteq \A \o{\rtimes} G$. On the other hand, $\A \rtimes_r G$ is a quotient of $\A \rtimes_f G$. There is a faithful conditional expectation $\bbE:\A \rtimes_r G \to \A$ such that for any finitely-supported function $a:G \to \A$,
\[
    \bbE\left(\sum_g a_gg\right) = a_e.
\]
If $\A$ is a von Neumann algebra, $\bbE$ extends uniquely to a faithful normal conditional expectation $\bbE:\A \o{\rtimes} G \to \A$.  Details of these constructions and facts may be found in~\cite[Chapter 7]{PedersenBook}.

\subsection{Unique Extension Properties for $C^*$-Inclusions}

The condition of having a unique pseudo-expectation is one of several ``unique extension properties''  for a  $C^*$-inclusion which have received attention of late.  Our main results show that for $C^*$-inclusions of the form \[\A\subseteq \A\rtimes_r G,\] certain unique extension properties imply norming.   In this section, we briefly review the unique extension properties which arise in this article.
\begin{definition}\label{uextdefs}
\begin{enumerate}
\item A $C^*$-inclusion $\A \subseteq \B$ has the \textbf{pure extension property} (\textsf{PEP}) if every pure state on $\A$ extends uniquely to a pure state on $\B$.
  \medskip
\item A useful relaxation of the \textsf{PEP} is the \textbf{almost extension property} (\textsf{AEP}) of Nagy and Reznikoff \cite[Section 2]{NagyReznikoff2014}. Instead of insisting that every pure state on $\A$ extends uniquely to a pure state on $\B$, the requirement is that a weak-$*$ dense collection of pure states on $\A$ extend uniquely to pure states on $\B$. Obviously, the \textsf{PEP} implies the \textsf{AEP}.
  \medskip
\item A \textbf{conditional expectation} for a $C^*$-inclusion $\A \subseteq \B$ is a unital completely positive (UCP) map $E:\B \to \A$ such that $E|_{\A} = \id$. Since we are not assuming that $\A$ is injective, there is no reason to expect that the identity map $\id:\A \to \A$ will have a UCP extension $E:\B \to \A$. In fact, if $\B$ is injective but $\A$ is not, then no such extension is possible. So there are many $C^*$-inclusions with no conditional expectations at all. We will denote by $\text{\textsf{CE}}_{\leq 1}$ the property of having \underline{at most one} conditional expectation.
\medskip\item\label{uextdefs4} Hamana proved that every $C^*$-algebra $\A$ has an \textbf{injective envelope} $I(\A)$ \cite[Theorem 4.1]{Hamana1979}. This is an injective $C^*$-algebra containing $\A$, and such that the only intermediate injective operator system $\A \subseteq \S \subseteq I(\A)$ is $I(\A)$ itself. The injective envelope is unique up to a $*$-isomorphism that fixes $\A$ pointwise. By injectivity, the inclusion map $\iota:\A \to I(\A)$ must have a UCP extension $\theta:\B \to I(\A)$. Pitts called such extensions \textbf{pseudo-expectations} in \cite{Pitts2017}, because they generalize conditional expectations. When a $C^*$-inclusion admits a unique pseudo-expectation, we say that it has the \textbf{unique pseudo-expectation property} (\textsf{!PsExp}), and if the unique pseudo-expectation happens to be faithful, we say that the $C^*$-inclusion has the \textbf{faithful unique pseudo-expectation property} (\textsf{f!PsExp}). Clearly, \textsf{!PsExp} implies $\text{\textsf{CE}}_{\leq 1}$.
\end{enumerate}
\end{definition}
It has been known for some time that the unique extension properties described above satisfy the following (in general, strict) hierarchy:
\[
    \text{\textsf{PEP}} \implies \text{\textsf{AEP}} \implies \text{\textsf{!PsExp}} \implies \text{\textsf{CE}}_{\leq 1}
\]
Only the middle implication is non-trivial. We provide an elementary proof in Appendix \ref{A}, although it also follows from Kwa\'{s}niewski and Meyer's work on aperiodicity, which we now discuss.

\begin{definition}\label{aperiod}
  Following \cite[Definition 2.1]{KwasniewskiMeyer2022}, we say that a $C^*$-inclusion $\A \subseteq \B$ is \textbf{aperiodic} if for every $b \in \B$, $\eps > 0$, and $\D \in \bbH(\A)$ (the non-zero hereditary subalgebras of $\A$), there exists $d \in \D_+^1$ (the norm-one positive elements of $\D$) and $a \in \A$ such that $\|dbd - a\| < \eps$.
\end{definition}
The almost extension property implies aperiodicity \cite[Theorem 5.5]{KwasniewskiMeyer2022}, which in turn implies uniqueness of pseudo-expectations \cite[Theorem 3.6]{KwasniewskiMeyer2022}. When $\B$ is separable, aperiodicity is equivalent to the almost extension property \cite[Theorem 5.5]{KwasniewskiMeyer2022}.

Summarizing:
\begin{figure}[!ht]
\boxed{
\xymatrix{
    \textsf{PEP} \ar@{=>}[r] & \textsf{AEP} \ar@{=>}[r] & \text{aperiodicity} \ar@{=>}[r] \ar@/^2pc/@{=>}[l]^(0.5){\text{separable}} & \textsf{!PsExp} \ar@{=>}[r] & \text{\textsf{CE}}_{\leq 1}
}}
\caption{Hierarchy of Unique Extension Properties}
\label{fig1}
\end{figure}
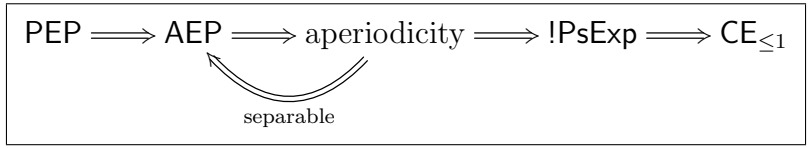

\subsection{Discrete Crossed Products} \label{DCP}

For the $C^*$-inclusion $\A \subseteq \A \rtimes_r G$ arising from a discrete $C^*$-dynamical system $(\A,G,\alpha)$, the unique extension properties from the previous section have been characterized in terms of the underlying dynamics (see Theorem \ref{discrete crossed products} below). In order to state the result, we will need various notions of outerness and freeness for $*$-automorphisms. Our understanding of these topics benefitted tremendously from \cite[Section 2]{KwasniewskiMeyer2018} and \cite[Chapter 8]{PedersenBook}.

\begin{definition}\label{outercond}
We say that $\alpha \in \Aut(\A)$:
\begin{enumerate}
\item is \textbf{inner} if $\alpha = \Ad(u)$ for some unitary $u \in \A$;
\item is \textbf{outer} if it is not inner;
\item has \textbf{no (non-zero) dependent elements} if the only $x \in \A$ such that
\[
    \text{$\alpha(a)x = xa$ for all $a \in \A$}
\]
is $x = 0$ (any such $x$ is called a ``dependent element'' for $\alpha$ \cite[Definition 1.3]{Kallman1969});
\item is \textbf{properly outer} if for every $\J \in \bbI^\alpha(\A)$ (the non-zero $\alpha$-invariant ideals of $\A$) and each unitary $u \in M(\J)$ (the multiplier algebra of $\J$), we have that $\|\alpha|_{\J} - \Ad(u)\| = 2$ \cite[Definition 2.1]{Elliott1980};
\item\label{outercond5} satisfies \textbf{Kishimoto's condition} if for all $\eps > 0$ and $\D \in \bbH(\A)$, there exists $d \in \D_+^1$ such that $\|d\alpha(d)\| < \eps$ \cite[Theorem 2.1]{Kishimoto1982}.
\end{enumerate}
\end{definition}

Clearly ``no dependent elements'' and ``properly outer'' are stronger forms of outerness. This is also the case for Kishimoto's condition. Indeed, if $\alpha = \Ad(u)$ satisfies Kishimoto's condition, then for all $\eps > 0$ and $\D \in \bbH(A)$, there exists $d \in \D_+^1$ such that
\[
    \|dud\| = \|dudu^*\| = \|d\alpha(d)\| < \eps.
\]
But then by the following lemma, $\|u\| \leq 2\eps$ for all $\eps > 0$, a contradiction.

\begin{lemma}
Let $\A$ be a unital $C^*$-algebra, $a \in \A$, and $\eps>0$. Suppose that for all $\D \in \bbH(\A)$ there exists $d \in \D_+^1$ such that $\|dad\| \leq \eps$. Then $\|a\| \leq 2\eps$.
\end{lemma}

\begin{proof}
First let us show that when $a=a^*$ satisfies the hypotheses of the Lemma, then $\|a\|\leq \eps$.   Arguing by contradiction, assume $\|a\|>\eps$.   Replacing $a$ with $-a$ if necessary, we may assume $\|a\| \in \sigma(a)$ (the spectrum of $a$). Define $\delta = \frac{\|a\| + \eps}{2}$. Then $p = \chi_{(\delta,\infty)}(a) \in \A^{**}$ is a non-zero projection such that $ap \geq \delta p$. Let $\D := (p\A^{**}p) \cap \A$.  Then $\D$ is hereditary, and we now show it is non-zero.  Indeed, if $f \in C(\bbR,[0,1])$ satisfies $f((-\infty,\delta]) = \{0\}$ and $f(\|a\|) = 1$, then $0 \neq f(a) \in \A$ and $f(a)p = f(a)$.  So $\D\neq \{0\}$.

By assumption, there exists $d \in \D_+^1$ such that $\|dad\| \leq \eps$. But then
\[
    \eps \geq \|dad\| = \|dapd\| \geq \delta\|dpd\| = \delta\|d\|^2 = \delta,
\]
a contradiction. Thus $\|a\| \leq \eps$.

Now let $a \in \A$ be arbitrary. If $\D \in \bbH(\A)$ and $d \in \D_+^1$ satisfies $\|dad\| \leq \eps$, then
\[
    \left\|d\left(\frac{a+a^*}{2}\right)d\right\|, ~ \left\|d\left(\frac{a-a^*}{2i}\right)d\right\| \leq \eps.
\]
By the self-adjoint case above,
\[
    \|a\| \leq \left\|\frac{a+a^*}{2}\right\| + \left\|\frac{a-a^*}{2i}\right\| \leq 2\eps.
\]
\end{proof}

Kishimoto's condition has a number of non-trivial reformulations. Indeed, the following are equivalent:
\begin{enumerate}
\item $\alpha$ satisfies Kishimoto's condition;
\vskip 3pt\item for all $\J \in \bbI^\alpha(\A)$, the Borchers spectrum $\Gamma_{\text{Bor}}(\alpha|_{\J})$ is non-trivial (i.e., does not equal $\{1\} \subseteq \bbT$) \cite[Theorem 2.1]{Kishimoto1982};
\vskip 3pt\item $I(\alpha) \in \Aut(I(\A))$ is outer on any non-trivial invariant central summand, where $I(\alpha)$ is the unique extension of $\alpha$ to a $*$-automorphism of the injective envelope $I(\A)$ of $\A$ \cite[Theorem 7.4 and Remark 7.5]{Hamana1985};
\vskip 3pt\item $I(\alpha) \in \Aut(I(\A))$ has no dependent elements \cite[Proposition 5.1]{Hamana1982}.
\end{enumerate}
Because of these equivalences, Kishimoto's condition goes by several different names in the literature, including ``spectral non-triviality'' and (unfortunately!) ``properly outer''. Given our focus on pseudo-expectations, the last two of the above reformulations will be the most useful.\\

In general, we have the following implications:
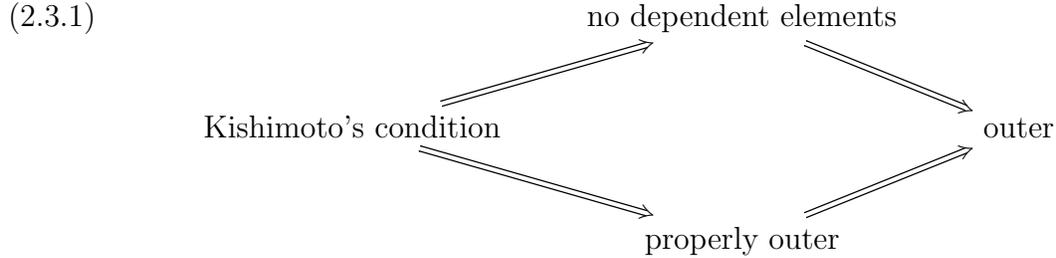
\begin{equation}\label{variousimps}
\xymatrix{
    & \text{no dependent elements} \ar@{=>}[dr] & \\
    \text{Kishimoto's condition} \ar@{=>}[ur]\ar@{=>}[dr] & & \text{outer}\\
    & \text{properly outer} \ar@{=>}[ur] & \\
}
\end{equation}
The fact that Kishimoto's condition implies proper outerness was observed by Kishimoto \cite[Remark 2.5]{Kishimoto1982}. We provide a slightly different proof in Appendix \ref{C}. The fact that Kishimoto's condition implies no dependent elements can be deduced from \cite[Lemma 2.2]{Zarikian2019a}, for example.

We extend these definitions to discrete $C^*$-dynamical systems by insisting they hold pointwise. More precisely, we say that a discrete $C^*$-dynamical system $(\A,G,\alpha)$ is outer (resp. has no dependent elements, is properly outer, satisfies Kishimoto's condition) if $\alpha_g$ is outer (resp. has no dependent elements, is properly outer, satisfies Kishimoto's condition) for all $g \in G\setminus\{e\}$.

A discrete $C^*$-dynamical system $(\A,G,\alpha)$ gives rise to a discrete topological dynamical system $(\mathcal{\widehat{A}},G,\hat{\alpha})$, where $\mathcal{\widehat{A}}$ is the spectrum of $\A$ (cf. \cite[Chapter 3]{Dixmier1977} or \cite[Chapter 4]{PedersenBook}) and for each $g \in G$, $\hat{\alpha}_g \in \Homeo(\mathcal{\widehat{A}})$ is defined by
\[
    \text{$\hat{\alpha}_g([\pi]) = [\pi \circ \alpha_g^{-1}]$ for all $[\pi] \in \mathcal{\widehat{A}}$.}
\]
If $\alpha_g$ is inner, then $\hat{\alpha}_g = \id_{\mathcal{\widehat{A}}}$. 
Thus, in addition to outerness conditions on the original action $G \curvearrowright \A$, we may also consider freeness conditions on the induced action $G \curvearrowright \mathcal{\widehat{A}}$. Even though we are only concerned with $(\mathcal{\widehat{A}},G,\hat\alpha)$, we state the definitions for general topological dynamical systems.

\begin{definition}\label{freecond}
Let $G \curvearrowright X$ be an action of a discrete group on a topological space by homeomorphisms. For $g \in G$, let
\[
    \Fix(g) = \{x \in X: g \cdot x = x\}
\]
be the fixed points of $g$. We say that $G \curvearrowright X$ is:
\begin{enumerate}
\item \textbf{free} if $\Fix(g) = \emptyset$ for all $g \in G\setminus\{e\}$;
\item \textbf{essentially free} if $G \curvearrowright D$ is free for some dense $G$-invariant set $D \subseteq X$;
\item \textbf{topologically free} if $\Fix(g)^\circ = \emptyset$ for all $g \in G\setminus\{e\}$.
\end{enumerate}
\end{definition}

It is easy to see that
\[
    \text{freeness} \implies \text{essential freeness} \implies \text{topological freeness}
\]
\begin{remark}\label{tfr->esfr}
If $X$ is Hausdorff and $G$ is countable, then topological freeness implies essential freeness. Indeed, by the Baire Category Theorem,
\[
    D := \bigcap_{g \in G\setminus\{e\}} \Fix(g)^c
\]
is a dense $G$-invariant subset of $X$ such that $G \curvearrowright D$ is free.
\end{remark}

It remains to compare the freeness properties of $(\mathcal{\widehat{A}},G,\hat{\alpha})$ with the outerness properties of $(\A,G,\alpha)$. Archbold and Spielberg observed that if $G \curvearrowright \mathcal{\widehat{A}}$ is topologically free, then $G \curvearrowright \A$ is properly outer \cite[Proposition 1]{ArchboldSpielberg1994}. In fact:
\[
    \text{topologically free} \implies \text{Kishimoto's condition}
\]
So the weakest of the freeness properties implies the strongest of the outerness properties. We provide one proof of this fact in Appendix \ref{B}. It also follows from \cite[Corollary 4.8]{KwasniewskiMeyer2022}.

Here is a summary of the dynamical conditions and their relationships. 
\begin{figure}[!ht]
\boxed{
\xymatrix{
    & & & & \parbox{1in}{\center no dependent\\ elements} \ar@{=>}[dr] & \\
    \text{free} \ar@{=>}[r] & \parbox{0.7in}{\center essentially\\ free} \ar@{=>}[r] & \parbox{0.9in}{\center topologically\\ free} \ar@{=>}[r] & \parbox{0.9in}{\center Kishimoto's\\ condition} \ar@{=>}[ur]\ar@{=>}[dr] & & \text{outer}\\
    & & & & \parbox{0.6in}{\center properly\\ outer} \ar@{=>}[ur] & \\
}}
\caption{Hierarchy of Dynamical Conditions}
\label{fig2}
\end{figure}
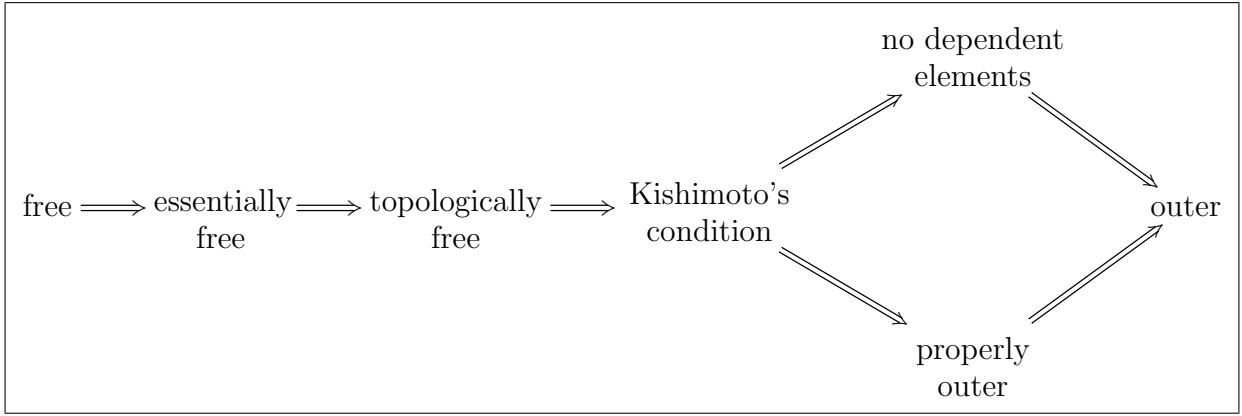

It should come as no surprise that for particular classes of $C^*$-algebras, certain of the outerness and freeness conditions discussed above might be equivalent. This will be important for us, so we record some of what is known and, to the best of our knowledge, provide original references.

\begin{theorem} \label{equivalences}
Let $G \curvearrowright \A$ be an action of a discrete group on a unital $C^*$-algebra.
\begin{enumerate}
\item[i.] If $\A$ is \textbf{abelian}, then
\[
    \text{topologically free} \iff \text{Kishimoto's condition} \iff \text{no dependent elements} \iff \text{properly outer};
\]
\item[ii.] If $\A$ is \textbf{separable}, then
\[
    \text{topologically free} \iff \text{Kishimoto's condition} \iff \text{properly outer};
\]
\item[iii.] If $\A$ is \textbf{simple}, then
\[
    \text{Kishimoto's condition} \iff \text{no dependent elements} \iff \text{properly outer} \iff \text{outer};
\]
\item[iv.] If $\A$ is a \textbf{von Neumann algebra}, then
\[
    \text{Kishimoto's condition} \iff \text{no dependent elements} \iff \text{properly outer}.
\]
\end{enumerate}
\end{theorem}

\begin{proof}
(i) \cite[Theorem 1]{EnomotoTamaki1975} and \cite[Remark (ii) following Proposition 1]{ArchboldSpielberg1994}.

(ii) \cite[Theorem 6.6]{OlesenPedersen1982}.

(iii) \cite[Theorem 4.2]{Olesen1975} (see also \cite[Corollary 8.9.10]{PedersenBook}.

(iv) \cite[Theorem 1.11]{Kallman1969} and \cite[discussion before Theorem 2.1]{Kishimoto1982}.
\end{proof}

\begin{remark}
For von Neumann algebras, the equivalent conditions in Theorem \ref{equivalences} (iv) all amount to the usual von Neumann-algebraic notion of ``proper outerness'', namely that the automorphisms are not inner on any non-zero invariant central summand.
\end{remark}

With this preparation, we  now state:

\begin{theorem} \label{discrete crossed products}
Let $(\A,G,\alpha)$ be a discrete $C^*$-dynamical system. Then
\begin{enumerate}
\item[i.] $\A \subseteq \A \rtimes_r G$ has the \textsf{PEP} $\iff$ $G \curvearrowright \mathcal{\widehat{A}}$ is free \cite[Corollary 2.5]{Zarikian2019b} (see also \cite[Theorem 2]{AkemannWeaver2004});
\item[ii.] $\A \subseteq \A \rtimes_r G$ has the \textsf{AEP} $\iff$ $G \curvearrowright \mathcal{\widehat{A}}$ is essentially free \cite[Corollary 2.6]{Zarikian2019b};
\item[iii.] $\A \subseteq \A \rtimes_r G$ has a unique pseudo-expectation $\iff$ $\A \subseteq \A \rtimes_r G$ is aperiodic \cite[Proposition 2.4]{KwasniewskiMeyer2022} $\iff$ $G \curvearrowright \A$ satisfies Kishimoto's condition \cite[Theorem 3.5]{Zarikian2019a};
\item[iv.] $\A \subseteq \A \rtimes_r G$ has a unique conditional expectation $\iff$ $\A^c = Z(\A)$ (the relative commutant equals the center) $\iff$ $G \curvearrowright \A$ has no dependent elements \cite[Theorem 3.2]{Zarikian2019a}.
\end{enumerate}
\end{theorem}

\subsection{Norming} \label{norming section}

For the reader's convenience,  we next collect the norming results on which this paper relies.

Recall from~\cite[Equation~2.12]{PopSinclairSmith2000} that a $C^*$-algebra $\A \subseteq B(\H)$ is called \textbf{locally cyclic} if for all $\eps > 0$ and $\xi_1, \xi_2, ..., \xi_n \in \H$, there exists $\xi \in \H$ and $a_1, a_2, ..., a_n \in \A$ such that $\|a_i\xi - \xi_i\| < \eps$, $1 \leq i \leq n$.   Examples of locally cyclic subalgebras in $B(\H)$  include  MASAs and von Neumann algebras of types ${\rm I}_\infty$, ${\rm II}_\infty$, and ${\rm III}$.  

\begin{theorem}[\cite{PopSinclairSmith2000}, Theorem 2.7] \label{locally cyclic}
Let $\A \subseteq B(\H)$ be a (not necessarily unital) $C^*$-algebra. Then $\A$ norms $B(\H)$ if and only if $\A$ is locally cyclic.  In particular, when $\A$ is a MASA in $B(\H)$, $\A$ norms $B(\H)$.
\end{theorem}

\begin{theorem}[\cite{PopSinclairSmith2000}, Theorem 2.9] \label{vNa norms B(H)}
Let $\M \subseteq B(\H)$ be a von Neumann algebra of type I$_\infty$, II$_\infty$, or III. Then $\M$ norms $B(\H)$.
\end{theorem}

\begin{theorem}[\cite{PopSinclairSmith2000}, Theorem 3.2] \label{no fin dim reps}
Let $\A \subseteq \B$ be a $C^*$-inclusion, where $\A$ has no finite-dimensional representations and $\B$ is finitely-generated as a left $\A$-module. Then $\A$ norms $\B$.
\end{theorem}

\begin{theorem}[\cite{CameronSmith2016}, Theorem 5.1] \label{vN crossed product norming}
Let $G$ be a discrete group acting on a von Neumann algebra $\M$ by properly outer $*$-automorphisms. Then $\M$ norms $\M \overline{\rtimes} G$.
\end{theorem}

\begin{remark}
Since $\M \rtimes_r G \subseteq \M \o{\rtimes} G$, we see that $\M$ norms $\M \rtimes_r G$ whenever $G \curvearrowright \M$ is properly outer.
\end{remark}

\section{Main Results}

\setcounter{subsection}{1}
In this section we establish two (apparently) distinct norming results for the $C^*$-inclusion $\A \subseteq \A \rtimes_r G$, namely Theorems \ref{norming_I} and \ref{norming_II}. Before we begin, we pause to make a useful observation:

\begin{lemma} \label{countable reduction}
The $C^*$-inclusion $\A \subseteq \A \rtimes_r G$ is norming if and only if the $C^*$-inclusion $\A \subseteq \A \rtimes_r H$ is norming for every finitely-generated subgroup $H \subseteq G$.
\end{lemma}

\begin{proof}
Assume that $\A \subseteq \A \rtimes_r H$ is norming for every finitely-generated subgroup $H \subseteq G$. If $x \in \A \rtimes_r G$, then there exists $a_1, a_2, ..., a_k \in \A$ and $g_1, g_2, ..., g_k \in G$ such that
\[
    \left\|x - \sum_{i=1}^k a_ig_i\right\| < \eps.
\]
Letting $H = \langle g_1, g_2, ..., g_k \rangle \subseteq G$, we see that $x_0 := \sum_{i=1}^k a_ig_i \in \A \rtimes_r H$. Likewise, if $X \in M_n(\A \rtimes_r G)$, there exists a finitely-generated subgroup $H \subseteq G$ and $X_0 \in M_n(\A \rtimes_r H)$ such that $\|X - X_0\| < \eps$. By assumption, there exist $R \in \Ball(\ROW_n(\A))$ and $C \in \Ball(\COL_n(\A))$ such that $\|RX_0C\| > \|X_0\| - \eps$. Then
\[
    \|RXC\| > \|RX_0C\| - \eps > \|X_0\| - 2\eps > \|X\| - 3\eps.
\]
Thus $\A$ norms $\A \rtimes_r G$. The converse is elementary.
\end{proof}

\subsection{Essential Freeness implies Norming}

In this section, we show that $\A \subseteq \A \rtimes_r G$ is norming provided that $G \curvearrowright \mathcal{\widehat{A}}$ is essentially free. Equivalently, by Theorem \ref{discrete crossed products} (ii), the \textsf{AEP} implies norming for $C^*$-inclusions of the form $\A \subseteq \A \rtimes_r G$.

\begin{lemma}
Let $(\A,G,\alpha)$ be a discrete $C^*$-dynamical system, $\phi \in S(\A)$, and $g \in G$. For any extension $\varPhi \in S(\A \rtimes_r G)$ of $\phi$, the following representations of $\A$ are unitarily equivalent:
\begin{enumerate}
\item $\pi_\varPhi|_{\A}$ acting on $\H_g := \overline{\pi_\varPhi(\A g)\xi_\varPhi} \subseteq \H_\varPhi$;
\item $\pi_\phi \circ \alpha_g^{-1}$ acting on $\H_\phi$.
\end{enumerate}
\end{lemma}

\begin{proof}
For $a \in \A$, define
\[
    U_0\pi_\phi(a)\xi_\phi = \pi_\varPhi(\alpha_g(a)g)\xi_\varPhi.
\]
Since
\[
    \|\pi_\varPhi(\alpha_g(a)g)\xi_\varPhi\|^2
    = \varPhi(g^{-1}\alpha_g(a)^*\alpha_g(a)g) = \varPhi(a^*a) = \phi(a^*a) = \|\pi_\phi(a)\xi_\phi\|^2,
\]
we see that $U_0$ is well-defined and extends to a unitary operator $U:\H_\phi \to \H_g$. For $a, a_1 \in \A$, we have that
\begin{eqnarray*}
    U^*\pi_\varPhi(a)U\pi_\phi(a_1)\xi_\phi
    &=& U^*\pi_\varPhi(a)\pi_\varPhi(\alpha_g(a_1)g)\xi_\varPhi\\
    &=& U^*\pi_\varPhi(a\alpha_g(a_1)g)\xi_\varPhi\\
    &=& U^*\pi_\varPhi(\alpha_g(\alpha_g^{-1}(a)a_1)g)\xi_\varPhi\\
    &=& \pi_\phi(\alpha_g^{-1}(a)a_1)\xi_\phi\\
    &=& \pi_\phi(\alpha_g^{-1}(a))\pi_\phi(a_1)\xi_\phi.
\end{eqnarray*}
Thus $\pi_\varPhi(a) = U\pi_\phi(\alpha_g^{-1}(a))U^*$, $a \in \A$.
\end{proof}

\begin{lemma} \label{GNS direct sum}
Let $(\A,G,\alpha)$ be a discrete $C^*$-dynamical system and $\phi \in S(\A)$. Define $\varPhi := \phi \circ \bbE \in S(\A \rtimes_r G)$, where $\bbE:\A \rtimes_r G \to \A$ is the canonical faithful conditional expectation. Then the following representations of $\A$ are unitarily equivalent:
\begin{enumerate}
\item $\pi_\varPhi|_{\A}$ acting on $\H_\varPhi$;
\item $\bigoplus_{g \in G} (\pi_\phi \circ \alpha_g^{-1})$ acting on $\bigoplus_{g \in G} \H_\phi$.
\end{enumerate}
\end{lemma}

\begin{proof}
By the previous lemma, it suffices to show that
\[
    \H_\varPhi = \bigoplus_{g \in G} \H_g,
\]
where $\H_g := \overline{\pi_\varPhi(\A g)\xi_\varPhi} \subseteq \H_\varPhi$. Indeed, suppose $g_1, g_2 \in G$, with $g_1 \neq g_2$. Then for all $a_1, a_2 \in \A$, we have that
\begin{eqnarray*}
    \langle \pi_\varPhi(a_1g_1)\xi_\varPhi, \pi_\varPhi(a_2g_2)\xi_\varPhi \rangle
    &=& \varPhi(g_2^{-1}a_2^*a_1g_1)\\
    &=& \phi(\bbE(\alpha_{g_2}^{-1}(a_2^*a_1)g_2^{-1}g_1))\\
    &=& \phi(\alpha_{g_2}^{-1}(a_2^*a_1)\bbE(g_2^{-1}g_1))\\
    &=& 0.
\end{eqnarray*}
Thus $\H_{g_1} \perp \H_{g_2}$, so that $\bigoplus_{g \in G} \H_g \subseteq \H_\varPhi$. Now suppose $x \in \A \rtimes_r G$. Then there exist $a_1, a_2, ..., a_k \in \A$ and $g_1, g_2, ..., g_k \in G$ such that
\[
    \left\|x - \sum_{i=1}^k a_ig_i\right\| < \eps.
\]
It follows that
\[
    \left\|\pi_\varPhi(x)\xi_\varPhi - \sum_{i=1}^k \pi_\varPhi(a_ig_i)\xi_\varPhi\right\| < \eps,
\]
which implies $\H_\varPhi \subseteq \bigoplus_{g \in G} \H_g$.
\end{proof}

\begin{lemma} \label{equiv or disjoint}
Let $(\A,G,\alpha)$ be a discrete $C^*$-dynamical system and $\phi_1, \phi_2 \in PS(\A)$. For $i = 1, 2$, define $\varPhi_i := \phi_i \circ \bbE$, where $\bbE:\A \rtimes_r G \to \A$ is the canonical faithful conditional expectation. Then $\pi_{\varPhi_1}|_{\A}$ and $\pi_{\varPhi_2}|_{\A}$ are either equivalent or disjoint.
\end{lemma}

\begin{proof}
Suppose $\pi_{\varPhi_1}|_{\A}$ and $\pi_{\varPhi_2}|_{\A}$ are not disjoint. By Lemma \ref{GNS direct sum}, $\pi_{\varPhi_1}|_{\A} \cong \bigoplus_{g \in G} (\pi_{\phi_1} \circ \alpha_g^{-1})$ and $\pi_{\varPhi_2}|_{\A} \cong \bigoplus_{g \in G} (\pi_{\phi_2} \circ \alpha_g^{-1})$. It follows that for some $g_1, g_2 \in G$, $\pi_{\phi_1} \circ \alpha_{g_1}^{-1}$ and $\pi_{\phi_2} \circ \alpha_{g_2}^{-1}$ are not disjoint. Since $\pi_{\phi_1}, \pi_{\phi_2}$ are irreducible, $\pi_{\phi_1} \circ \alpha_{g_1}^{-1} \cong \pi_{\phi_2} \circ \alpha_{g_2}^{-1}$. Then $\bigoplus_{g \in G} (\pi_{\phi_1} \circ \alpha_g^{-1}) \cong \bigoplus_{g \in G} (\pi_{\phi_2} \circ \alpha_g^{-1})$, which implies $\pi_{\varPhi_1}|_{\A} \cong \pi_{\varPhi_2}|_{\A}$.
\end{proof}

For a $C^*$-inclusion $\A \subseteq \B$, we denote by $PS(\A \uparrow \B)$ the pure states on $\A$ which extend uniquely to (necessarily pure) states on $\B$.

\begin{lemma} \label{GNS norming}
Let $(\A,G,\alpha)$ be a discrete $C^*$-dynamical system and $\phi \in PS(\A \uparrow \A \rtimes_r G)$. If $\tilde{\phi} \in PS(\A \rtimes_r G)$ is the unique state extension of $\phi$, then $\pi_{\tilde{\phi}}(\A)$ norms $B(\H_{\tilde{\phi}})$.
\end{lemma}

\begin{proof}
By uniqueness, $\tilde{\phi} = \phi \circ \bbE$, where $\bbE:\A \rtimes_r G \to \A$ is the canonical faithful conditional expectation. By the previous lemma and \cite[Lemma 2.3 (ii)]{PopSinclairSmith2000}, it suffices to show that the inclusion
\[
    \left(\left(\bigoplus_{g \in G} (\pi_\phi \circ \alpha_g^{-1})\right)(\A)\right)'' \subseteq B\left(\bigoplus_{g \in G} \H_\phi\right)
\]
is norming. By \cite[Theorem 2.4]{Zarikian2019b}, $\{\pi_\phi \circ \alpha_g^{-1}: g \in G\}$ is a mutually disjoint family of irreducible representations (see also \cite[Theorem 2]{AkemannWeaver2004}). Thus
\[
    \left(\left(\bigoplus_{g \in G} (\pi_\phi \circ \alpha_g^{-1})\right)(\A)\right)''
    = \bigoplus_{g \in G} ((\pi_\phi \circ \alpha_g^{-1})(\A))''
    = \bigoplus_{g \in G} B(\H_\phi).
\]
Since $\bigoplus_{g \in G} B(\H_\phi) \subseteq B\left(\bigoplus_{g \in G} \H_\phi\right)$ is locally cyclic, the result now follows by an application of Theorem \ref{locally cyclic}.
\end{proof}

\begin{theorem} \label{norming_I}
Let $(\A,G,\alpha)$ be a discrete $C^*$-dynamical system. If $G \curvearrowright \mathcal{\widehat{A}}$ is essentially free (equivalently, if $\A \subseteq \A \rtimes_r G$ has the \textsf{AEP}), then $\A$ norms $\A \rtimes_r G$.
\end{theorem}

\begin{proof}
By Theorem \ref{discrete crossed products} (ii), $\A \subseteq \A \rtimes_r G$ has the almost extension property, therefore a unique pseudo-expectation, necessarily the canonical faithful conditional expectation $\bbE:\A \rtimes_r G \to \A$. For each $\phi \in PS(\A \uparrow \A \rtimes_r G)$, let $\tilde{\phi} = \phi \circ \bbE \in PS(\A \rtimes_r G)$ be its unique state extension. Choose $\F \subseteq PS(\A \uparrow \A \rtimes_r G)$ such that $\{\pi_{\tilde{\phi}}|_{\A}: \phi \in \F\}$ is a maximal mutually disjoint family, and define
\[
  \tilde{\pi} := \bigoplus_{\phi \in \F} \pi_{\tilde{\phi}}: \A \rtimes_r G \to \bigoplus_{\phi \in \F} B(\H_{\tilde{\phi}}).  \]

We claim that $\tilde{\pi}$ is faithful on $\A$. Indeed, suppose $a \in \A_+\setminus\{0\}$ and $\tilde{\pi}(a) = 0$. Since $PS(\A \uparrow \A \rtimes_r G)$ is weak-$*$ dense in $PS(\A)$, there exists $\rho \in PS(\A \uparrow \A \rtimes_r G)$ such that $\rho(a) \neq 0$. Then $\pi_\rho(a) \neq 0$, which implies $\pi_{\tilde{\rho}}(a) \neq 0$. By Lemma \ref{equiv or disjoint}, $\pi_{\tilde{\rho}}$ and $\pi_{\tilde{\phi}}$ are disjoint for all $\phi \in \F$, which contradicts the maximality of $\F$.  Thus the claim holds.

By \cite[Theorem 3.5]{PittsZarikian2015}, the inclusion $\A \subseteq \A \rtimes_r G$ is essential. Thus, by \cite[Proposition 3.3]{PittsZarikian2015}, $\tilde{\pi}$ is faithful on $\A \rtimes_r G$. By \cite[Lemma 2.3 (ii)]{PopSinclairSmith2000}, to show that $\A$ norms $\A \rtimes_r G$, it suffices to show that $\tilde{\pi}(\A)''$ norms $\tilde{\pi}(\A \rtimes_r G)$. In fact, by Lemma \ref{GNS norming} and \cite[Corollary 6.5]{PittsZarikian2015}, $\tilde{\pi}(\A)''$ norms $\bigoplus_{\phi \in \F} B(\H_{\tilde{\phi}})$, since \[
    \tilde{\pi}(\A)'' = \left(\left(\bigoplus_{\phi \in \F} \pi_{\tilde{\phi}}\right)(\A)\right)'' = \bigoplus_{\phi \in \F} (\pi_{\tilde{\phi}}(\A))''.
\]
\end{proof}

As a consequence of this result, we get:

\begin{corollary}\label{separable}
Let $G \curvearrowright \A$ be an action of a discrete group on a separable unital $C^*$-algebra $\A$. If $\A \subseteq \A \rtimes_r G$ has a unique pseudo-expectation, then $\A$ norms $\A \rtimes_r G$.
\end{corollary}

\begin{proof}
By Lemma \ref{countable reduction} and \cite[Proposition 2.6]{PittsZarikian2015}, we may assume that $G$ is finitely-generated. Then $\A \rtimes_r G$ is separable, and so $\A \subseteq \A \rtimes_r G$ satisfies the \textsf{AEP} by \cite[Theorem 5.5]{KwasniewskiMeyer2022} (see also \cite[Lemma 2.17]{KwasniewskiMeyer2018}, which in turn references the proof of \cite[Proposition 4.4]{OlesenPedersen1980}). Thus $\A$ norms $\A \rtimes_r G$ by Theorem \ref{norming_I}.
\end{proof}

Also we are able to give a different proof of the following result, which follows from \cite[Theorems 6.15 and 8.2]{Pitts2017}. Note that this is not covered by the previous corollary, since we do not assume that $X$ is metrizable.

\begin{corollary}[\cite{Pitts2017}]\label{1hom}
Let $G \curvearrowright X$ be an action of a discrete group on a compact Hausdorff space. If $C(X) \subseteq C(X) \rtimes_r G$ has a unique pseudo-expectation (equivalently, if $G \curvearrowright X$ is topologically free), then $C(X)$ norms $C(X) \rtimes_r G$.
\end{corollary}

\begin{proof}
Assume that $G \curvearrowright X$ is topologically free. By Lemma \ref{countable reduction}, we may assume that $G$ is finitely-generated. By Remark~\ref{tfr->esfr}, $G \curvearrowright X$ is essentially free, and so $C(X)$ norms $C(X) \rtimes_r G$ by Theorem \ref{norming_I}.
\end{proof}

Corollary~\ref{1hom} can be extended to the setting where $C(X)$ is replaced by an $n$-homogeneous $C^*$-algebra $\A$ (that is, a $C^*$-algebra whose non-zero irreducible representations are all $n$-dimensional), and the remainder of the section is devoted to doing this. We begin with some notation. 

For vectors $\xi, \eta$ in the Hilbert space $\H$, we use $\xi\eta^*$ for the rank one operator $\H \ni x \mapsto \innerprod{x,\eta}\xi$. Then $(\xi_1\eta_1^*)(\xi_2\eta_2^*)=\innerprod{\xi_2,\eta_1}\xi_1\eta_2^*$, the adjoint of $(\xi\eta^*)$ is $\eta\xi^*$, and $\norm{\xi\eta^*}=\norm{\xi}\, \norm{\eta}$.  Also $\ker(\xi\eta^*)=\{\eta\}^\perp$ and $\ran(\xi\eta^*)=\bbC \xi$. It follows that $\xi\eta^*=\xi_1\eta_1^*$ if and only if there are scalars $s, t$ such that
$\xi_1=s \xi$, $\eta_1=t\eta$ and $s\overline t=1$. 

We omit the proof of the following well-known fact.

\begin{lemma}[cf.\ {\cite[Lemma~8.7.4]{PedersenBook}}] \label{autobh} 
Suppose $\H$ is a Hilbert space and $\alpha\in\aut(\bh)$. There exists a unitary $U\in\bh$, which is unique up to a scalar $\lambda\in \bbT$, such that $\alpha(T)=UTU^*$. Furthermore, if $\zeta\in \H$ is a fixed unit vector and $\xi\in \H$ is a unit vector chosen so that $\alpha(\zeta\zeta^*)=\xi\xi^*$, the unitary $U$ may be taken to be 
\begin{equation}\label{autodef}
    U\eta=\alpha(\eta\zeta^*)\xi, \quad \eta\in \H.
\end{equation} 
\end{lemma}

The next proposition follows from the discussion by Kadison and Ringrose in~\cite[Example~(d), p.\ 55]{KadisonRingrose1967}, which was elaborated by Smith in~\cite[Theorem~3.5]{SmithMiSoo1970}.  We provide a proof for convenience.

\begin{proposition} \label{conU} 
Let $X$ be a topological space, fix $z\in X$, and suppose $x\mapsto \alpha_x$ is a point-norm continuous mapping of $X$ into $\aut(\bh)$ (i.e.\ for every $T\in\bh$ and $x\in X$, $\lim_{y\rightarrow x}\norm{\alpha_x(T)-\alpha_y(T)}=0$). Then there exists an open neighborhood $N$ of $z$ and a strongly continuous mapping $N\ni x\mapsto U_x\in \U(\bh)$ such that for every $x\in N$ and $T\in\bh$, \[\alpha_x(T)=U_xTU_x^*.\]
\end{proposition}

\begin{proof} Without loss of generality, we may assume
  $\alpha_z=\id$.  Fix a unit vector $e\in\H$ and a real number
  $\delta$ with $0<\delta<1$.  The map $x\mapsto \alpha_x(ee^*)e$ is a
  continuous mapping of $X$ into $\H$, so
\[h(x):=\innerprod{\alpha_x(ee^*)e,e}^{1/2}\] is a continuous function
of $X$ into the unit interval.  In particular, 
$N=h^{-1}(\delta,1]$ is an open subset of $X$ containing $z$.

For each $x\in N$, there is a unit vector $\eta_x\in \H$ such that
$\alpha_x(ee^*)=\eta_x\eta_x^*$. Note that
$h(x)=|\innerprod{e,\eta_x}|$, so  there is a unique $\lambda_x\in\bbT$ such that
$\lambda_x\innerprod{\eta_x,e}>0$.  Put
\[\xi_x:=\lambda_x \eta_x,\] so that $\alpha_x(ee^*)=\xi_x\xi_x^*$ and
$\innerprod{\xi_x,e}>0$.
Then
\[\norm{\alpha_x(ee^*)e}^2=\norm{\xi_x\xi_x^*e}^2=|\innerprod{e,\xi_x}|^2=
  \innerprod{e,\xi_x}^2,\] whence
$N\ni x\mapsto \innerprod{e,\xi_x}$ is a continuous function.  It
follows that
\[N\ni x\mapsto \xi_x=\frac{1}{\innerprod{e,\xi_x}} \alpha_x(ee^*) e\] is a
  continuous function of $N$ into the unit sphere of $\H$.

For $x\in N$ and $f\in \H$, let
\[U_xf=\alpha_x(fe^*)\xi_x.\] By Lemma~\ref{autobh}, $U_x$ is a unitary such that
$\alpha_x=\Ad_{U_x}$.

To show strong continuity of $x\mapsto U_x$ on $N$,
choose a unit vector $f\in \H$, $x\in N$, and $\eps>0$.  Let $\O \subseteq N$ be a relatively open neighborhood of $x$
such that for $y\in \O$,
\[\max\{\norm{\alpha_x(fe^*)-\alpha_y(fe^*)}, \norm{\xi_x-\xi_y}\}<\eps/2.\]   Then for any $y\in \O$,

\begin{eqnarray*}
  \norm{(U_x-U_y)f}&=&\norm{\alpha_x(fe^*)\xi_x-\alpha_y(fe^*)\xi_y}\\
&\leq&
    \norm{\alpha_x(fe^*)(\xi_x-\xi_y)}+\norm{\alpha_x(fe^*)\xi_y-\alpha_y(fe^*)\xi_y}\\
  &\leq& \norm{\xi_x-\xi_y}+ \norm{\alpha_x(fe^*)-\alpha_y(fe^*)} <\eps.
\end{eqnarray*}
\end{proof}

\begin{proposition} \label{tf1} Suppose $X$ is a compact Hausdorff space, $\H$
  is a finite-dimensional Hilbert space, and $\alpha$
  is an automorphism of $\A:=C(X)\otimes\bh$ such that $\hat\alpha$ does not act topologically freely on $\mathcal{\widehat{A}}$.  Then $\alpha$ has a non-zero dependent element.
\end{proposition}
\begin{proof}  View $\A$ as the set of continuous functions from $X$
  into $\bh$.  For any subset $N\subseteq X$, write $\A_N:=\{f|_N:
  f\in \A\}$.

  Since $\hat\alpha$ is not
topologically free, there exists an open set $N_0\subseteq \mathcal{\widehat{A}}=X$ such
that the action of $\hat\alpha$ on $\mathcal{\widehat{A}}_{N_0}$ is the identity on $\mathcal{\widehat{A}}_{N_0}$.
This means that $\alpha$ fixes the center of $\A_{N_0}$.  So for each
$x\in N_0$, there exists $\alpha_x\in \aut(\bh)$ such that for $f\in \A$,
\[\alpha(f)(x)=\alpha_x(f(x)).\]    Let $T\in \bh$.
Letting $f_T\in \A$ be the (constant) function $X\ni x\mapsto T$, we find $\alpha_x(T)=\alpha(f_T)(x)$, so that the map $N_0\ni x\mapsto \alpha_x(T)$ is norm-continuous.

Let $z\in N_0$.  Since $\dim\H<\infty$, an application of Proposition~\ref{conU} yields an open
subset $N\subseteq N_0$ with $z\in N$ and a norm-continuous function 
$N\ni x\mapsto U_x\in \U(\bh)$ such that for every $x\in N$,
\[\alpha(f)(x)=U_xf(x)U_x^*.\]  Since $N$ is an open subset of $X$,
there exists a non-zero (scalar-valued) $h\in C_0(N)$.   For $x\in X$, let 
\[d(x):=\begin{cases} h(x)U_x &\text{if $x\in N$;}\\
    0&\text{if $x\in X\setminus N$.}
  \end{cases}
  \]
Then $d\in \A$, and for any $f\in \A$ and $x\in N$,
$\alpha(f)(x)d(x)=U_xf(x)U_x^* U_xh(x)=d(x)f(x)$, and when $x\in
X\setminus N$, we have $\alpha(f)(x)d(x)=0=d(x)f(x)$.  So for $f\in \A$,
\[\alpha(f)d=df,\] whence $d$ is a dependent element for $\alpha$.  This completes
the proof.
\end{proof}
\begin{corollary}\label{nhomcor}
Suppose $\A$ is an $n$-homogeneous \cstaralg\ and $\alpha$
is an automorphism of $\A$ such that $\hat\alpha$ does not act topologically freely on
$\mathcal{\widehat A}$.  Then $\alpha$ has a non-zero dependent element.
\end{corollary}
\begin{proof}
Let $\H$ be a Hilbert space with $\dim(\H)=n$ and put $X=\mathcal{\widehat{A}}$.  A result of Fell~\cite[Theorem~3.2]{Fell1961} shows that $X$ is compact and Hausdorff, the center of $\A$ is $C(X)$, and $\A$ is the set of continuous cross sections of a
  locally trivial $\bh$ bundle $\B$ over $X$.

 Since the action of $\hat\alpha$ on $\mathcal{\widehat A}$ is not topologically free, there exists a non-empty open set $\O\subseteq X$ such that $\hat\alpha|_\O=\id|_\O$.  Local triviality of $\B$ implies that   we may assume the restriction of $\B$ to $\overline \O$ is trivial, so that $\A|_{\overline\O}\simeq C(\overline\O)\otimes\bh$.   As done previously, we identify $C(\overline\O)\otimes\bh$ with $C(\overline\O,\bh)$. 
 Proposition~\ref{tf1} gives a non-zero dependent element $a$ for the restriction of $\alpha$ to $\A_{\overline\O}$.   Let $S:=\{x\in \overline\O:  a(x)\neq 0\}$.  This is an open set, so we may find a non-zero element $f\in C_0(\O)\subseteq C(X)\subseteq \A$.  Then $d:=fa\in\A$ is a non-zero dependent element for $\alpha$.
\end{proof}

Here is the promised extension of Corollary~\ref{1hom}.

\begin{theorem}\label{nhom} Let $\A$ be an $n$-homogeneous \cstaralg\ and suppose $G \curvearrowright\A$ is an action of the discrete group $G$ on $\A$.  If $\A\subseteq \A\rtimes_r G$ has a unique pseudo-expectation, then $\A$ norms $\A\rtimes_r G$.
\end{theorem}

\begin{proof} 
As noted previously, we may assume $G$ is a countable group.  Theorem~\ref{discrete crossed products}(iii) shows that $G\curvearrowright\A$ satisfies Kishimoto's condition, and as observed earlier (see Diagram~\ref{variousimps}), $\A$ has no dependent elements.
By Corollary~\ref{nhomcor}, the action of $G$ is topologically free. 

Since $\mathcal{\widehat{A}}$ is Hausdorff,  the topological freeness of the action of $G$ on $\mathcal{\widehat{A}}$ implies that the action is essentially free (see Remark~\ref{tfr->esfr}).  By Theorem~\ref{norming_I}, $\A$ norms $\A\rtimes_r G$.
\end{proof}

\subsection{Residual Proper Outerness implies Norming}

An individual automorphism $\alpha$ of the $C^*$-algebra $\A$ is called \textbf{residually properly outer} if the induced automorphism $\dot{\alpha} \in \Aut(\A/\J)$ is properly outer for any invariant ideal $\J \subsetneq \A$. (It follows easily that $\dot{\alpha}$ is actually residually properly outer.) An action $G \curvearrowright \A$ is \textbf{residually properly outer} if $\alpha_g$ is residually properly outer for all $g \in G\setminus\{e\}$.

In this section, we show that $\A \subseteq \A \rtimes_r G$ is norming provided that $G \curvearrowright \A$ is residually properly outer.
While the proofs in the previous section relied heavily on GNS-type arguments and Theorem \ref{locally cyclic}, the proofs in this section have a distinctly von Neumann-algebraic flavor, and ultimately appeal to Theorem \ref{vN crossed product norming}. In Appendix \ref{D}, we compare residual proper outerness and essential freeness, showing that the former is not implied by the latter. This shows that the results of this section do not imply the results of the previous section. In our concluding remarks, we will discuss the possibility of unifying our norming results.

\begin{lemma} \label{lem1}
Let $\A \subseteq B(\H)$ be an abelian von Neumann algebra, and for $n \in \bbN$ fixed, let $\{T_\lambda\}_{\lambda \in \Lambda} \subseteq \A \otimes \bbM_n \subseteq B(\H \otimes \bbC^n)$ be a uniformly bounded net converging strongly to $0$. Then, given $\varepsilon > 0$, there exist $\lambda_0 \in \Lambda$ and a nonzero projection $z \in \A$ so that $\|T_{\lambda_0}(z \otimes I_n)\| < \varepsilon$.
\end{lemma}

\begin{proof}
Choose $\delta>0$ so that \[\sqrt{2}n^{5/2}\delta^{1/2} < \varepsilon.\] Suppose initially that $T_\lambda \geq 0$ for all $\lambda$. Choose a unit vector $\eta \in \H$ and let $e_1, e_2, ..., e_n$ be the standard basis for $\bbC^n$. Then, by strong convergence, choose $\lambda_0$ so that
\begin{equation}\label{eq1}
    \|T_{\lambda_0}(\eta \otimes e_i)\| < \delta, ~ 1 \leq i \leq n.
\end{equation}
Now write $T_{\lambda_0}=(f_{ij})$ with $f_{ij} \in \A$. It follows from \eqref{eq1} that
\begin{equation}\label{eq2}
    \|f_{ii}\eta\| < \delta, ~ 1 \leq i \leq n,
\end{equation}
so if we define $f = \sum_{i=1}^n f_{ii} \in \A$ then
\begin{equation}\label{eq3}
    \|f\eta\| < n\delta
\end{equation}
from \eqref{eq2}. Let $z \in \A$ be the spectral projection of $f$ for the interval $[0,2n\delta]$. If $z = 0$ then $f \geq 2n\delta$ so
\begin{equation}\label{eq4}
    \langle f\eta,\eta \rangle \geq 2n\delta,
\end{equation}
and so $\|f\eta\| \geq 2n\delta$, contradicting \eqref{eq3}. Thus $z \neq 0$.

Now $T_{\lambda_0}(z \otimes I_n)=(f_{ij}z) \geq 0$ so has a square root $S=(s_{ij}) \in \A \otimes \bbM_n$. Then
\begin{equation}\label{eq5}
    f_{ii}z=\sum_{k=1}^n s_{ki}^*s_{ki}, ~ 1 \leq i \leq n,
\end{equation}
so
\begin{equation}\label{eq6}
    \|s_{ki}\|^2 \leq \|f_{ii}z\| \leq 2n\delta, ~ 1 \leq i,k \leq n.
\end{equation}
Thus
\begin{equation}\label{eq7}
    \|S\| \leq \sum_{k,i}\|s_{ki}\| \leq \sqrt{2}n^{5/2}\delta^{1/2},
\end{equation}
so
\begin{equation}\label{eq8}
    \|T_{\lambda_0}(z \otimes I_n)\| = \|S\|^2 \leq 2n^5\delta.
\end{equation}

For the general case, apply the preceding calculation to the positive net $T_\lambda^*T_\lambda \to 0$ strongly. We obtain $\lambda_0$ and a nonzero projection $z \in \A$ so that
\begin{equation}\label{eq9}
    \|T_{\lambda_0}^*T_{\lambda_0}(z \otimes I_n)\| \leq 2n^5\delta,
\end{equation}
so
\begin{equation}\label{eq10}
    \|T_{\lambda_0}(z \otimes I_n)\| \leq \sqrt{2}n^{5/2}\delta^{1/2}<\eps.
\end{equation}
\end{proof}

\begin{lemma}\label{lem2}
Let $\A \subseteq B(\H)$ be a unital $C^*$-algebra such that $\M := \A''$ is type I$_n$. If $\alpha \in \Aut(\A)$ is residually properly outer and $\o{\alpha} \in \Aut(\M)$ is an extension, then $\o{\alpha}$ is outer.
\end{lemma}

\begin{proof}
Suppose that $\o{\alpha} = \Ad(u)$ for some unitary $u \in \M$. By the Kaplansky Density Theorem, there exists a net $\{u_i\} \subseteq \A$ of unitaries such that $u_i \to u$ strongly. By Lemma \ref{lem1}, there exists a non-zero central projection $z \in \M$ and an index $i_0$ such that $\|(u - u_{i_0})z\| < 1$. Note that $\o{\alpha}(z) = z$ and let $\J = \{a \in \A: az = 0\}$.  Then $\J$ is a proper $\alpha$-invariant ideal of $\A$. It is easy to see that under the canonical identification $\A/\J \cong \A z$, the induced automorphism $\dot{\alpha} \in \Aut(\A/\J)$ corresponds to the restricted automorphism $\o{\alpha}|_{\A z} \in \Aut(\A z)$. Then $\o{\alpha}|_{\A z}$ is (residually) properly outer. But $\o{\alpha}|_{\A z} = \Ad(uz)|_{\A z}$ and $\|\Ad(uz)|_{\A z} - \Ad(u_{i_0}z)\| < 2$, a contradiction.
\end{proof}

\begin{remark}
The assumption that $\M$ in the previous lemma is type I$_n$ is crucial. Indeed, by \cite[Corollary 3.2]{Phillips1987}, there exists a unitary $u$ in the group ${\rm II}_1$ factor $L(\bbF_2)$ such that $\Ad(u)$ restricts to an outer automorphism of $C_r^*(\bbF_2)$. Since $C_r^*(\bbF_2)$ is simple, outer automorphisms are in fact residually properly outer.
\end{remark}

The previous lemma can be bootstrapped to prove the following stronger result:

\begin{lemma} \label{lem3}
Let $\A \subseteq B(\H)$ be a unital $C^*$-algebra such that $\M := \A''$ is type I$_n$. If $\alpha \in \Aut(\A)$ is residually properly outer and $\o{\alpha} \in \Aut(\M)$ is an extension, then $\o{\alpha}$ is properly outer.
\end{lemma}

\begin{proof}
Let $z \in \M\setminus\{0\}$ be an $\o{\alpha}$-invariant central projection. Then, as in the proof of Lemma \ref{lem2}, $\o{\alpha}|_{\A z}$ is residually properly outer. Since $\A z \subseteq B(z\H)$ is a unital $C^*$-algebra such that $(\A z)'' = \M z$ is type I$_n$, $\o{\alpha}|_{\M z}$ is outer by Lemma \ref{lem2}. Thus $\o{\alpha}$ is properly outer.
\end{proof}

The following lemma and its proof are very close to \cite[Remark 2.5(ii)]{PopSmith2008}.

\begin{lemma}\label{guts} Let $\A$ be a $C^*$-algebra and let $(\A,G,\alpha)$ be a $C^*$-dynamical system, where $G$ is a discrete group.  If $\A$ has no finite-dimensional representations, then $\A$ norms $\A^{**} \o{\rtimes} G$. A fortiori, $\A$ norms $\A \rtimes_r G$.
\end{lemma}

\begin{proof} 
We assume initially that $\A$ is a von Neumann algebra $\M$ with no finite-dimensional representations.  Write the standard basis vectors of $\ell^2(G)$ as $\{\delta_g:g\in G\}$ 
and denote by $e_{g,h}$ the matrix unit that satisfies $e_{g,h}\delta_h=\delta_g$ for $g,h\in G$.
If $\M$ is represented on a Hilbert space $\mathcal H$, then $\M$ is also represented on $\mathcal{H}\otimes\ell^2(G)$ by
\[\pi(x)=\sum_{g\in G}\alpha_{g^{-1}}(x)\otimes e_{g,g},\ \ \ x\in \M,\]
and $G$ is represented on $\mathcal{H}\otimes\ell^2(G)$ by
\[\lambda(g)(\eta\otimes \delta_h)=\eta\otimes\delta_{gh},\ \ \ \eta\in \mathcal{H},\ \ g,h\in G.\]
Then $\M\overline{\rtimes} G\subseteq \M\overline{\otimes}B(\ell^2(G))$, so it suffices to show that $\pi(\M)$ norms the latter algebra. For each finite subset $F\subseteq G$, let $p_F$ be the projection onto span$\{\delta_h:h\in F\}$. Then the compression of $\M\overline{\otimes}B(\ell^2(G))$ by $1\otimes p_F$ has the form $\M\otimes \mathbb{M}_k$ where $k$ is the cardinality of $F$. Note that $1\otimes p_F\in \pi(\M)'$.  Now
\[(1\otimes p_F)\pi(x)(1\otimes p_F)=\sum_{g\in F}\alpha_{g^{-1}}(x)\otimes e_{g,g},\]
so
$(1\otimes p_F)(\M\overline{\otimes}B(\ell^2(G)))(1\otimes p_F)$ is a finitely-generated bimodule over the von Neumann algebra $(1\otimes p_F)(\pi(\M))(1\otimes p_F)$. Thus the larger algebra is normed by the smaller one from \cite[Cor. 3.3]{PopSinclairSmith2000}, using that $\M$ has no finite-dimensional representations.

Now consider $X\in \mathbb{M}_n(\M\overline{\otimes} B(\ell^2(G)))$, $\|X\|=1$. Given $\varepsilon >0$, choose $F$ so large that 
\[\|(1\otimes p_F\otimes I_n)X(1\otimes p_F\otimes I_n)\| > 1-\varepsilon.\]
Then choose $C \in {\mathrm{Col}}_n((1\otimes p_F)\pi(\M))$ with $\|C\|=1$ so that
\[\|(1\otimes p_F\otimes I_n)X(1\otimes p_F\otimes I_n)C\|>1-\varepsilon.\] 
Since $C(1\otimes p_F)=(1\otimes p_F\otimes I_n)C$, we obtain $\|XC\| > 1-\varepsilon$. This proves that $\M$ column-norms $\M\overline{\rtimes} G$, and norming now follows from \cite[Lemma 2.4]{PopSinclairSmith2000}. 

Finally, consider a $C^*$-algebra $\A$ with no finite-dimensional representations. Then the von Neumann algebra $\M:=\A^{**}$ also fails to have finite-dimensional representations. From the first part of the proof, $\M$ norms $\M\overline\rtimes G$. In its universal representation, $\A$ is strongly dense in $\A^{**}=\M$ and so, by \cite[Lemma 2.3(ii)]{PopSinclairSmith2000}, $\A$ norms $\M\overline\rtimes G$ and hence also norms the subalgebra  $\A\rtimes_r G$.
\end{proof}

\begin{theorem} \label{norming_II}
Let a discrete group $G$ act on a unital $C^*$-algebra $\A$ by residually properly outer automorphisms. Then $\A$ norms $\A \rtimes_r G$.
\end{theorem}

\begin{proof}
Each $\alpha_g$ extends uniquely to an automorphism of $\A^{**}$, denoted $\o{\alpha}_g$. We first show that $\A^{**}$ norms $\A^{**} \o{\rtimes} G$. Each $\o{\alpha}_g$ respects the type decomposition of $\A^{**}$, so let $z$ be a central projection so that $\A^{**}z$ is one of the summands.  A routine argument shows that it suffices to examine each type separately.

Suppose first that $\A^{**}z$ is type I$_n$ for a fixed integer $n \geq 1$. Then for each $g \in G\setminus\{e\}$, $\o{\alpha}_g$ is properly outer by Lemma \ref{lem3}. By \cite[Theorem 5.1]{CameronSmith2016}, $\A^{**}z$ norms $(\A^{**}z) \o{\rtimes} G$.
Lemma~\ref{guts} shows $\A^{**}z$ norms $(\A^{**}z)\o{\rtimes} G$ for the other types.

We conclude that $\A^{**}$ norms $\A^{**} \o{\rtimes} G$, so $\A$ also norms this crossed product by \cite[Lemma 2.3(ii)]{PopSinclairSmith2000}. A fortiori, $\A$ norms $\A \rtimes_r G$.
\end{proof}

The previous result gives the following immediate corollary.

\begin{corollary}\label{simple}
If $G \curvearrowright \A$ is an outer action of a discrete group on a simple $C^*$-algebra, then $\A$ norms $\A \rtimes_r G$.
\end{corollary}

\section{Concluding Remarks and Open Questions}
\setcounter{subsection}{1}

We have shown that the inclusion $\A \subseteq \A \rtimes_r G$ is norming provided either (1) $G \curvearrowright \mathcal{\widehat{A}}$ is essentially free or (2) $G \curvearrowright \A$ is residually properly outer. It follows that when $\A \subseteq \A \rtimes_r G$ has a unique pseudo-expectation, the inclusion is norming provided $\A$ is:
\begin{enumerate}
\item abelian \cite[Theorem~8.2]{Pitts2017};
\item $n$-homogeneous (Theorem~\ref{nhom});
\item separable (Corollary~\ref{separable});
\item simple (Corollary~\ref{simple}); or
\item a von Neumann algebra \cite[Theorem~5.1]{CameronSmith2016}.
\end{enumerate}
We are left to wonder if this might be true in general:
\begin{question} \label{Q1}
For an arbitrary discrete $C^*$-dynamical system $(\A,G,\alpha)$, is the $C^*$-inclusion $\A \subseteq \A \rtimes_r G$ norming whenever it has a (faithful) unique pseudo-expectation?
\end{question}

A proof might go like this: If $\A \subseteq \A \rtimes_r G$ has a faithful unique pseudo-expectation, then $G \curvearrowright \A$ satisfies Kishimoto's condition (see Theorem \ref{discrete crossed products}), which implies the extended action $G \curvearrowright I(\A)$ also satisfies Kishimoto's condition (see the reformulations of Kishimoto's condition in Section \ref{DCP}). Since $I(\A)$ is an injective $C^*$-algebra (and thus an $AW^*$-algebra), one should be able to show that $I(\A)$ norms $I(\A) \rtimes_r G$, \`a~la Theorem \ref{vN crossed product norming}. Since $\A$ and $I(\A)$ are closely related, perhaps then $\A$ also norms $I(\A) \rtimes_r G$, \`a~la \cite[Lemma 2.3 (ii)]{PopSinclairSmith2000}. Indeed, the inclusion $\A \subseteq I(\A)$ is \textbf{rigid}, in the sense that the only complete contraction $u:I(\A) \to I(\A)$ such that $u|_{\A} = \id$ is the identity itself.

We carry out the first step of this program in Theorem \ref{IA} below. At the moment, we do not see how to carry out the second step of the program, so we leave it as an open question (Question \ref{Q2} below).

\begin{lemma} \label{lemma_A}
Let $\B$ be a unital $C^*$-algebra and $\{z_n\} \subseteq \B$ be mutually orthogonal central projections such that $\sup_{\B} z_n = 1$. Then the map $\B \to \prod_{n=1}^\infty \B z_n: b \mapsto (bz_n)_{n=1}^\infty$ is a $*$-monomorphism, therefore completely isometric. Thus for all $X \in M_k(\B)$, $\|X\| = \sup_n \|X(I_k \otimes z_n)\|$.
\end{lemma}

\begin{proof}
The specified map is clearly a $*$-homomorphism. Now suppose $b \in \Ball(\B_+)$ satisfies $bz_n = 0$ for all $n \in \bbN$. Then for all $n \in \bbN$,
\[
    z_n = (1-b)z_n = (1-b)^{1/2}z_n(1-b)^{1/2} \leq 1-b.
\]
Since $\sup_{\B} z_n = 1$, we conclude that $b = 0$.
\end{proof}

\begin{lemma} \label{lemma_B}
Let $G \curvearrowright \A$ be the action of a discrete group on a unital $C^*$-algebra. Suppose that $\{z_n\} \subseteq \A$ are mutually orthogonal $G$-invariant central projections such that $\sup_{\A} z_n = 1$. Then $\{z_n\} \subseteq \A \rtimes_r G$ are mutually orthogonal central projections such that $\sup_{\A \rtimes_r G} z_n = 1$.
\end{lemma}

\begin{proof}
To see that $z_n$ is central in $\A \rtimes_r G$, note that
\[
    z_nag = az_ng = agg^{-1}z_ng = ag\alpha_g^{-1}(z_n) = agz_n.
\]
Now let $x \in \A \rtimes_r G$ be an upper bound for $\{z_n\}$, and suppose $x \leq 1$. Denote by $\bbE:\A \rtimes_r G \to \A$ the canonical faithful conditional expectation. Then $\bbE(x) \in \A$ is an upper bound for $\{z_n\}$, and $\bbE(x) \leq 1$. It follows that $\bbE(x) = 1$, which implies $\bbE(1-x) = 0$, which in turn implies $x = 1$ (since $\bbE$ is faithful). Thus $\sup_{\A \rtimes_r G} z_n = 1$.
\end{proof}

\begin{lemma} \label{lemma_C}
Let $G \curvearrowright \A$ be the action of a discrete group on a unital $C^*$-algebra. Suppose that $\{z_n\} \subseteq \A$ are mutually orthogonal $G$-invariant central projections such that $\sup_{\A} z_n = 1$. If $\A z_n$ norms $\A z_n \rtimes_r G$ for all $n \in \bbN$, then $\A$ norms $\A \rtimes_r G$.
\end{lemma}

\begin{proof}
Let $X \in M_k(\A \rtimes_r G)$ and $\varepsilon > 0$. By Lemmas \ref{lemma_A} and \ref{lemma_B}, there exists $n \in \bbN$ such that
\[
    \|X(I_k \otimes z_n)\| > \|X\| - \varepsilon.
\]
Since $\A z_n$ norms $\A z_n \rtimes_r G$, there exist $R \in \Ball(\ROW_k(\A z_n))$ and $C \in \Ball(\COL_k(\A z_n))$ such that
\[
    \|RX(I_k \otimes z_n)C\| > \|X(I_k \otimes z_n)\| - \varepsilon.
\]
Then
\[
    \|RXC\| = \|RX(I_k \otimes z_n)C\| > \|X(I_k \otimes z_n)\| - \varepsilon > \|X\| - 2\varepsilon.
\]
Thus $\A$ norms $\A \rtimes_r G$.
\end{proof}

\begin{theorem}\label{IA}
Suppose that $G \curvearrowright \A$ satisfies Kishimoto's condition. Then $I(\A)$ norms $I(\A) \rtimes_r G$.
\end{theorem}

\begin{proof}
Just like von Neumann algebras, $AW^*$-algebras admit a type decomposition \cite[\slashS 15, Theorem 3]{BerberianBook}. Thus there are $G$-invariant central projections $z_{\fin}, z_{\inf} \in I(\A)$ such that $z_{\fin}+z_{\inf}=1$, $I(\A)z_{\fin}$ is finite, and $I(\A)z_{\inf}$ is properly infinite. By \cite[\slashS 17, Theorem 1 (2)]{BerberianBook}, $I(\A)z_{\inf}$ has no finite-dimensional representations, so $I(\A)z_{\inf}$ norms $I(\A)z_{\inf} \rtimes_r G$, according to Lemma \ref{guts}. Now $z_{\fin} = z_{\typeI_{\fin}}+z_{\typeII_{\fin}}$, where $z_{\typeI_{\fin}}, z_{\typeII_{\fin}} \in I(\A)$ are $G$-invariant central projections such that $I(\A)z_{\typeI_{\fin}}$ is type I (and finite) and $I(\A)z_{\typeII_{\fin}}$ is type II (and finite). By \cite[\slashS 19, Theorem 1]{BerberianBook}, $I(\A)z_{\typeII_{\fin}}$ has no finite-dimensional representations, so $I(\A)z_{\typeII_{\fin}}$ norms $I(\A)z_{\typeII_{\fin}} \rtimes_r G$, again using Lemma \ref{guts}. Finally, $z_{\typeI_{\fin}} = \sum_{n=1}^\infty z_{\typeI_n}$, where for each $n \in \bbN$, $z_n \in I(\A)$ is either $0$ or a $G$-invariant central projection such that $I(\A)z_n$ is type I$_n$ \cite[\slashS 18, Theorem 2]{BerberianBook}. If $I(\A)z_n$ is a type I$_n$ $AW^*$-algebra, then it is an $n$-homogeneous $C^*$-algebra \cite[\slashS 18, Proposition 2]{BerberianBook}. Since $G \curvearrowright I(\A)z_n$ satisfies Kishimoto's condition, $I(\A)z_n$ norms $I(\A)z_n \rtimes_r G$, by Theorem \ref{nhom}. Invoking Lemma \ref{lemma_C} completes the proof.
\end{proof}

\begin{question} \label{Q2}
If $I(\A)$ norms $I(\A) \rtimes_r G$, does $\A$ norm $I(\A) \rtimes_r G$ (or at least $\A \rtimes_r G$)?
\end{question}

\begin{question} \label{Q3}
More generally, if $\A \subseteq \B \subseteq B(\H)$ and a (completely isometric) copy of $I(\A)$ in $B(\H)$ norms $\B$, does $\A$ norm $\B$?
\end{question}

As indicated above, a positive solution to Question \ref{Q2} would imply a positive solution to Question \ref{Q1}.

\appendix

\section{Comparisons}

\subsection{Almost Extension Property implies Unique Pseudo-Expectation} \label{A}

\begin{lemma} \label{bijection}
Let $\A \subseteq \B$ be a $C^*$-inclusion and $\phi \in S(\A)$ be a state. Then there exists a bijective correspondence between the state extensions $\varPhi \in S(\B)$ of $\phi$ and the UCP extensions $\theta:\B \to B(\H_\phi)$ of the GNS representation $\pi_\phi:\A \to B(\H_\phi)$. In particular, $\phi$ has a unique state extension if and only if $\pi_\phi$ has a unique UCP extension.
\end{lemma}

\begin{proof}
If $\theta:\B \to B(\H_\phi)$ is a UCP extension of $\pi_\phi$, then the formula
\[
    \varPhi_\theta(b) = \langle \theta(b)\xi_\phi, \xi_\phi \rangle, ~ b \in \B,
\]
defines a state extension $\varPhi_\theta \in S(\B)$ of $\phi$. Conversely, if $\varPhi \in S(\B)$ is a state extension of $\phi$ with GNS representation $\pi_\varPhi:\B \to B(\H_\varPhi)$, then
\[
    \|\pi_\varPhi(a)\xi_\varPhi\|^2 = \varPhi(a^*a) = \phi(a^*a) = \|\pi_\phi(a)\xi_\phi\|^2, ~ a \in \A,
\]
and so there exists a linear isometry $V_\varPhi:\H_\phi \to \H_\varPhi$ such that
\[
    V_\varPhi\pi_\phi(a)\xi_\phi = \pi_\varPhi(a)\xi_\varPhi, ~ a \in \A.
\]
Then the formula
\[
    \theta_\varPhi(b) = V_\varPhi^*\pi_\varPhi(b)V_\varPhi, ~ b \in \B,
\]
defines a UCP extension $\theta_\varPhi:\B \to B(\H_\phi)$ of $\pi_\phi$. To see that these mappings are mutually inverse, note that if $\theta:\B \to B(\H_\phi)$ is a UCP extension of $\pi_\phi$, then for all $a_1, a_2 \in \A$,
\begin{eqnarray*}
    \langle \theta_{\varPhi_{\theta}}(b)\pi_\phi(a_1)\xi_\phi, \pi(a_2)\xi_\phi \rangle
    &=& \langle V_{\varPhi_{\theta}}^*\pi_{\varPhi_{\theta}}(b)V_{\varPhi_{\theta}}\pi_\phi(a_1)\xi_\phi, \pi_\phi(a_2)\xi_\phi \rangle\\
    &=& \langle \pi_{\varPhi_{\theta}}(b)\pi_{\varPhi_\theta}(a_1)\xi_{\varPhi_{\theta}}, \pi_{\varPhi_{\theta}}(a_2)\xi_{\varPhi_{\theta}} \rangle\\
    &=& \varPhi_{\theta}(a_2^*ba_1) = \langle \theta(a_2^*ba_1)\xi_\phi, \xi_\phi \rangle\\
    &=& \langle \pi_\phi(a_2)^*\theta(b)\pi_\phi(a_1)\xi_\phi, \xi_\phi \rangle\\
    &=& \langle \theta(b)\pi_\phi(a_1)\xi_\phi, \pi_\phi(a_2)\xi_\phi \rangle,
\end{eqnarray*}
so that $\theta_{\varPhi_{\theta}}(b) = \theta(b)$. Likewise, if $\varPhi \in S(\B)$ is a state extension of $\phi$, then
\[
    \varPhi_{\theta_{\varPhi}}(b) = \langle \theta_{\varPhi}(b)\xi_\phi, \xi_\phi \rangle
    = \langle V_\varPhi^*\pi_\varPhi(b)V_\varPhi\xi_\phi, \xi_\phi \rangle
    = \langle \pi_\varPhi(b)\xi_\varPhi, \xi_\varPhi \rangle = \varPhi(b).
\]
\end{proof}

\begin{theorem} \label{AEP implies !PsExp}
If the $C^*$-inclusion $\A \subseteq \B$ satisfies the almost extension property, then it has a unique pseudo-expectation.
\end{theorem}

\begin{proof}
Let $PS(\A \uparrow \B)$ denote the set of pure states on $\A$ which extend uniquely to pure states on $\B$. Assume that $PS(\A \uparrow \B)$ is weak-$*$ dense in $PS(\A)$. It follows that
\[
    \pi := \bigoplus_{\phi \in PS(\A \uparrow \B)} \pi_\phi:\A \to \bigoplus_{\phi \in PS(\A \uparrow \B)} B(\H_\phi)
\]
is a faithful $*$-homomorphism. Since $\bigoplus_{\phi \in PS(\A \uparrow \B)} B(\H_\phi)$ is an injective $C^*$-algebra, there exists a minimal injective operator system $\pi(\A) \subseteq \S \subseteq \bigoplus_{\phi \in PS(\A \uparrow \B)} B(\H_\phi)$. Now let $\theta:\B \to \S$ be a UCP map such that $\theta|_{\A} = \pi$. Then $\theta = \bigoplus_{\phi \in PS(\A \uparrow \B)} \theta_\phi$, where for each $\phi \in PS(\A \uparrow \B)$, $\theta_\phi:\B \to B(\H_\phi)$ is a UCP extension of $\pi_\phi$. By Lemma \ref{bijection}, each $\theta_\phi$ is uniquely determined. Thus $\theta$ is uniquely determined.
\end{proof}

\subsection{Topological Freeness implies Kishimoto's Condition} \label{B}

We begin with a technical lemma inspired by \cite[Lemma 2]{ArchboldSpielberg1994}.

\begin{lemma} \label{Archbold Spielberg lemma}
Let $(\A,G,\alpha)$ be a discrete $C^*$-dynamical system, $\pi:\A \to B(\H)$ be a non-zero irreducible representation, and $g \in G$. Then the following are equivalent:
\begin{enumerate}
\item[i.] $[\pi \circ \alpha_g^{-1}] \neq [\pi]$;
\item[ii.] If $\Theta:\A \rtimes_f G \to B(\H)$ is a UCP map extending $\pi$, then $\Theta(g) = 0$;
\item[iii.] If $\theta:\A \rtimes_r G \to B(\H)$ is a UCP map extending $\pi$, then $\theta(g) = 0$.
\end{enumerate}
\end{lemma}

\begin{proof}
(i $\implies$ ii) Suppose $\Theta:\A \rtimes_f G \to B(\H)$ is a UCP map extending $\pi$ and such that $\Theta(g) \neq 0$. Then
\[
    \Theta(g)\pi(a) = \Theta(ga) = \Theta(\alpha_g(a)g) = \pi(\alpha_g(a))\Theta(g), ~ a \in \A.
\]
By \cite[Lemma 2.1]{Zarikian2019b}, $[\pi \circ \alpha_g^{-1}] = [\pi]$.

(ii $\implies$ iii) Suppose $\theta:\A \rtimes_r G \to B(\H)$ is a UCP map extending $\pi$. Then $\theta \circ \lambda:\A \rtimes_f G \to B(\H)$ is a UCP map extending $\pi$, where $\lambda:\A \rtimes_f G \to \A \rtimes_r G$ is the unique $*$-epimorphism that fixes $C_c(G,\A)$. By assumption, $(\theta \circ \lambda)(g) = 0$. Thus $\theta(g) = \theta(\lambda(g)) = 0$.

(iii $\implies$ i) Suppose $[\pi \circ \alpha_g^{-1}] = [\pi]$. By the proof of the implication (ii $\implies$ iii) in \cite[Theorem 2.4]{Zarikian2019b}, there exists a UCP map $\theta:\A \rtimes_r G \to B(\H)$ extending $\pi$ and such that $\theta(g) \neq 0$.
\end{proof}

As mentioned in the main text, the following result is also a consequence of the much stronger result \cite[Corollary 4.8]{KwasniewskiMeyer2022}.

\begin{theorem} \label{top free implies Kishimoto}
If $G \curvearrowright \mathcal{\widehat{A}}$ is topologically free, then $G \curvearrowright \A$ satisfies Kishimoto's condition.
\end{theorem}

\begin{proof}
By Theorem \ref{discrete crossed products} (iii), it suffices to show that $\A \subseteq \A \rtimes_r G$ has a unique pseudo-expectation. To that end, let $\theta:\A \rtimes_r G \to I(\A)$ be a pseudo-expectation for $\A \subseteq \A \rtimes_r G$. Fix $g \in G\setminus\{e\}$. By assumption, $\Fix(g)^c$ is dense in $\mathcal{\widehat{A}}$. Let $\{\pi_i: i \in I_g\}$ be a family of non-zero irreducible representations containing one element from each unitary equivalence class in $\Fix(g)^c$, and define
\[
    \pi_g := \bigoplus_{i \in I_g} \pi_i:\A \to \bigoplus_{i \in I_g} B(\H_i).
\]
Then $\pi_g$ is a faithful $*$-homomorphism. Indeed, for $a \in \A$,
\[
    \pi_g(a) = 0 \implies \Fix(g)^c \subseteq \{[\pi] \in \mathcal{\widehat{A}}: \|\pi(a)\| = 0\},
\]
and the latter set is closed in $\mathcal{\widehat{A}}$ by \cite[Proposition 3.3.2]{Dixmier1977}. Since $\bigoplus_{i \in I_g} B(\H_i)$ is an injective $C^*$-algebra, we have that $I(\pi_g(\A)) \subseteq \bigoplus_{i \in I_g} B(\H_i)$ (as an operator subsystem). By the uniqueness of the injective envelope, there exists a complete order isomorphism $\iota_g:I(\A) \to I(\pi_g(\A))$ such that ${\iota_g}|_{\A} = \pi_g$. Then $\iota_g \circ \theta:\A \rtimes_r G \to I(\pi_g(\A))$ is a UCP map extending $\pi_g$. We may write $\iota_g \circ \theta = \bigoplus_{i \in I_g} \theta_i$, where each $\theta_i:\A \rtimes_r G \to B(\H_i)$ is a UCP map extending $\pi_i$. By Lemma \ref{Archbold Spielberg lemma}, $\theta_i(g) = 0$, since $g \cdot [\pi_i] \neq [\pi_i]$. Thus $\iota_g(\theta(g)) = 0$, which implies $\theta(g) = 0$. Since the choice of $g$ was arbitrary, $\theta = \bbE$, the canonical faithful conditional expectation of $\A \rtimes_r G$ onto $\A$.
\end{proof}

\subsection{Kishimoto's Condition implies Proper Outerness} \label{C}

\begin{theorem}[\cite{Kishimoto1982}, Remark 2.5]
Let $\alpha \in \Aut(\A)$. If $\alpha$ satisfies Kishimoto's condition, then $\alpha$ is properly outer.
\end{theorem}

\begin{proof}
Suppose there exists a non-zero $\alpha$-invariant ideal $\J \subseteq \A$ and a unitary $u \in M(\J)$ such that $\|\alpha|_{\J} - \Ad(u)\| < 2$. By \cite[Lemma 1.3]{SaitoWright1984}, we may assume that $M(\J) \subseteq I(\J)$. By \cite[Theorem 8.7.7]{PedersenBook}, there exists a $*$-derivation $\delta:\J \to \J$ such that $\alpha|_{\J} = \Ad(u) \circ \exp(\delta)$. By \cite[Theorem 2.1]{HamanaOkayasuSaito1982}, there exists a unique $*$-derivation $I(\delta):I(\J) \to I(\J)$ extending $\delta$. By \cite[Theorem 2]{Olesen1974}, there exists $h \in I(\J)_{sa}$ such that $I(\delta)(x) = i(hx-xh)$ for all $x \in I(\J)$. By \cite[Exercise 10.5.61]{KadisonRingroseVolII}, $\exp(I(\delta)) = \Ad(v)$, where $v = e^{ih} \in I(\J)$. Then $I(\alpha|_{\J}) = \Ad(uv)$, by uniqueness, which implies $\Gamma_{\text{Bor}}(\alpha|_{\J}) = \{1\}$ by \cite[Theorem 7.4]{Hamana1985}.
\end{proof}

\subsection{Essential Freeness vs.\ Residual Proper Outerness} \label{D}

Recall that $G \curvearrowright \mathcal{\widehat{A}}$ is \text{essentially free} if $G \curvearrowright D$ is free on a dense $G$-invariant set $D \subseteq \mathcal{\widehat{A}}$, and that $G \curvearrowright \A$ is \text{residually properly outer} if for any $g \in G\setminus\{e\}$ and any $\alpha_g$-invariant ideal $\J \subsetneq \A$, the induced automorphism $\dot{\alpha}_g \in \Aut(\A/\J)$ is properly outer. We claim that the following implications hold:
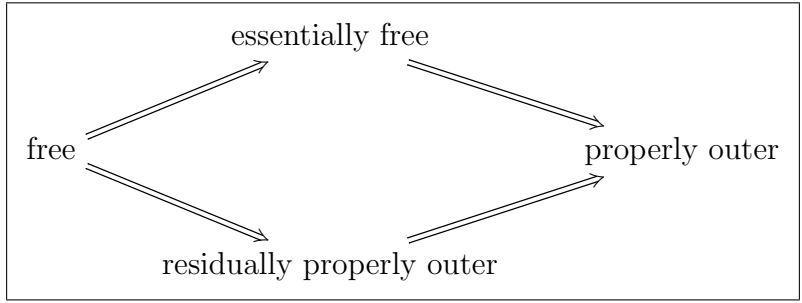
\[
\boxed{
\xymatrix{
    & \text{essentially free} \ar@{=>}[dr] & \\
    \text{free} \ar@{=>}[ur]\ar@{=>}[dr] & & \text{properly outer}\\
    & \text{residually properly outer} \ar@{=>}[ur] & \\
}}
\]
Only the implication ``free'' implies ``residually properly outer'' requires further explanation. Indeed, if $\alpha \in \Aut(\A)$ and $\J \subsetneq \A$ is an $\alpha$-invariant ideal, then by \cite[Proposition 3.2.1]{Dixmier1977}, $\hat{\dot{\alpha}} \in \Homeo(\mathcal{\widehat{A/J}})$ may be identified with $\hat{\alpha}|_{\mathcal{\widehat{A}}_{\J}} \in \Homeo(\mathcal{\widehat{A}}_{\J})$, where
\[
    \mathcal{\widehat{A}}_{\J} = \{[\pi] \in \mathcal{\widehat{A}}: \pi(\J) = \{0\}\},
\]
an $\hat{\alpha}$-invariant closed subset of $\mathcal{\widehat{A}}$. It follows that if $\hat{\alpha} \in \Homeo(\mathcal{\widehat{A}})$ is free, then $\hat{\dot{\alpha}}$ is free, which implies $\dot{\alpha}$ is properly outer. Thus $\alpha$ is residually properly outer when $\hat{\alpha}$ is free.\\

Now we examine the relationship between essential freeness and residual proper outerness.

When $\A$ is abelian, there are many ideals, and so residual proper outerness is a very strong condition. In fact, it is strictly stronger than essential freeness:

\begin{proposition}
Let $G \curvearrowright X$ be the action of a discrete group on a compact Hausdorff space. Then $G \curvearrowright C(X)$ is residually properly outer if and only if $G \curvearrowright X$ is free.
\end{proposition}

\begin{proof}
As noted at the beginning of this section, if $\hat{\alpha} \in \Homeo(X)$ is free, then $\alpha \in \Aut(C(X))$ is residually properly outer. Conversely, suppose $\alpha \in \Aut(C(X))$ is residually properly outer and consider the closed $\hat{\alpha}$-invariant subset of $X$, $K:= \Fix(\hat{\alpha})$. If $K \neq \emptyset$, then $\dot{\alpha} \in \Aut(C(K))$ is properly outer. But $\dot{\alpha} = \id_{C(K)}$, a contradiction. Thus $K = \emptyset$ and $\hat{\alpha}$ is free.
\end{proof}

On the other hand, when $\A$ is simple, there are no proper ideals, and residual proper outerness is just (proper) outerness. Thus in the simple case, essential freeness is stronger than residual proper outerness. Whether it is strictly stronger is unclear. If, in addition, $\A$ is separable and $G$ is countable, then residual proper outerness implies essential freeness, by Theorem \ref{equivalences} and \cite[Lemma 2.17]{KwasniewskiMeyer2018}. So if one believes that essential freeness is strictly stronger than residual proper outerness in the simple case, one needs to consider uncountable groups and/or non-separable $C^*$-algebras. Perhaps $\mathcal U(\bbM_2) \curvearrowright \O_2$ could be an example. Here $\mathcal U(\bbM_2)$ is the two-dimensional complex unitary group; $\O_2 = C^*(S_1,S_2)$ is the Cuntz algebra; and the automorphism associated with $\begin{bmatrix} u_{11} & u_{12}\\ u_{21} & u_{22} \end{bmatrix} \in \mathcal U(\bbM_2)$ is determined by the assignments $S_1 \mapsto u_{11}S_1+u_{12}S_2$ and $S_2 \mapsto u_{21}S_1+u_{22}S_2$. Of course, one could also search for examples among outer actions of discrete groups on II$_1$ factors.

In summary, we know that essential freeness does not imply residual proper outerness, but we do not know whether residual proper outerness implies essential freeness.

\end{document}